\documentclass[12pt]{amsart}

\usepackage[all]{xy}
\usepackage{graphics}
\usepackage{color}
\usepackage[T1]{fontenc}
\usepackage{parskip}
\usepackage[colorlinks]{hyperref}
\usepackage{enumitem}

\newcommand{\B}[1]{{\mathbf #1}}
\newcommand{\C}[1]{{\mathcal #1}}
\newcommand{\F}[1]{{\mathfrak #1}}

\newtheorem*{theorem}{Theorem}

\newtheorem{lemma}{Lemma}

\theoremstyle{definition}

\theoremstyle{remark}
\newtheorem*{remark}{Remark}

\numberwithin{figure}{section}
\numberwithin{table}{section}
\numberwithin{equation}{section}

\newcommand{\OP}{\operatorname}

\begin{document}

\title[Finite index subgroups in Chevalley groups are bounded]
{Finite index subgroups in Chevalley groups are bounded:\\
an addendum to ``On bi-invariant word metrics''}
\author{\'Swiatos\l aw R. Gal}
\address{Uniwersytet Wroc\l awski}
\email{Swiatoslaw.Gal@math.uni.wroc.pl}
\author{Jarek K\k{e}dra}
\address{University of Aberdeen and University of Szczecin}
\email{kedra@abdn.ac.uk}
\author{Alexander A. Trost}
\address{University of Aberdeen}
\email{r01aat17@abdn.ac.uk}


\begin{abstract}
We prove that finite index subgroups in $S$-arithmetic Chevalley
groups are bounded.
\end{abstract}

\maketitle

\section{Introduction}
A group $G$ is called {\em bounded} if every conjugation invariant norm on $G$
has finite diameter. Examples of bounded groups include $\OP{SL}(n,\B Z)$ for
$n\geq 3$, $\OP{Diff}_0(M)$, where $M$ is a manifold of dimension different
from $2$ and $4$, the commutator subgroup of Thompson's group $F$ and many
others.  A finite index subgroup of a bounded group does not have to be
bounded.  The simplest example is the infinite cyclic subgroup of the infinite
dihedral group. The purpose of this note is to prove the following result.

\begin{theorem}\label{T:main}
Let $G$ be a Chevalley group over the ring of $S$-integers in a number field
$\B k$ constructed from a root system whose irreducible components all have
rank at least $2$.  If $H\leq G$ is a subgroup of finite index then it is
bounded.
\end{theorem}

The above theorem generalises the main result of the paper \cite{MR2819193}.
The proof is similar with an additional ingredient being an explicit form of bounded
generation of a finite index subgroup of a boundedly generated group.
We also correct a couple of mistakes from the original proof.

\subsection*{Acknowledgement} 
The work was partially funded by a Leverhulme
Trust Research Project Grant RPG-2017-159 and the Polish National Science Centre grant
2017/27/B/ST1/01467.

\section{Definitions and known facts}
\subsection{Norm}\label{SS:norm}
A {\em conjugation invariant norm} of a group $G$ is a nonnegative function 
$\nu\colon G\to \B R$
such that the following conditions
\begin{enumerate}
\item $\nu(g)=0$ if and only if $g=1$,
\item $\nu\left( g^{-1} \right)=\nu(g)$,
\item $\nu(gh) \leq \nu(g)+\nu(h)$,
\item $\nu\left( h^{-1}gh \right)=\nu(g)$
\end{enumerate}
hold for every $g,h\in G$.

\subsection{Bounded group}\label{SS:bounded}
A group $G$ is called {\em bounded} if the diameter 
$$
\OP{diam}(G,\nu)=\sup\{\nu(g)\ |\ g\in G\} < \infty
$$ 
for every conjugation invariant norm $\nu$.  If $G$ is generated by finitely
many conjugacy classes then its boundedness is equivalent to the boundedness of
any conjugation invariant {\em word} norm \cite[Section 2.C]{MR2819193}.

\subsection{Arithmetic group}\label{SS:arithmetic}
Let $\B k$ be a number field (i.e. a finite extension of $\B Q$) and let 
$V_{\B k}$ denote the set of equivalence classes of valuations of $\B k$. Let
$S\subset V_{\B k}$ be a finite set containing all archimedean valuations.
The ring of $S$-integers is defined by
$$
\C O_S=\{x\in \B k\ | \ v(x)\geq 0 \text{ for all }v\notin S\}.
$$
Let $\B G$ be a connected algebraic group defined over $\B k$ with a fixed
embedding $\B G\to \OP{GL}(r)$.  A  subgroup of $\B G$ that is commensurable
with $\B G(\C O_S)= \B G \cap \OP{GL}(r,\C O_S)$ is called
an {\em $S$-arithmetic group} \cite[page 61]{MR1090825}.
More generally, an $S$-arithmetic group can be defined over any global field.

\subsection{Chevalley group \cite{eom-chevalley,MR3616493}}\label{SS:chevaley} 
Let $\F g$ be a semi-simple complex Lie algebra over $\B C$ with Cartan subalgebra 
$\F h$. Let ${\Phi}$ be a root system of $\F g$ with respect to $\F h$ with
simple roots
$\{\alpha_1,\dots,\alpha_k\}\subset \Phi$.
Let 
$$
\{H_{\alpha_i} (1\le i\le k); X_{\alpha} (\alpha\in\Phi)\}
$$ 
be a Chevalley basis of the algebra $\F g$, and let $\F g_{\B Z}$ be its linear
envelope over $\B Z$. Let ${\varphi}\colon \F g\to \F g\F l(r,\B C)$ be a faithful
representation.  There is a lattice
$L\subset \B C^r$ which is invariant with respect to all operators of the form
$\varphi(X_{\alpha})^m/m!$, where $m\in \B N$. If $\B k$ is
an arbitrary field then
homomorphisms $x_{\alpha}\colon \B (\B k,+) \to {\rm GL}(L\otimes {\B k})$ of the additive
group of $\B k$ into $\OP{GL}(L\otimes {\B k})$ are defined and given by the formulas
$$
x_{\alpha}(t) = \sum_{m=0}^\infty t^m\frac{\varphi(X_{\alpha})^m}{m!}.
$$ 
The subgroup $G(\Phi,\B k)\subset \OP{GL}(L\otimes {\B k})$ generated by 
$\{x_\alpha(t):\alpha\in\Phi, t\in~\B k\}$,
is called the {\em Chevalley group}
associated with the root system $\Phi$, the representation $\varphi$ and the field 
$\B k$. 

\subsection{Chevalley's commutator formula}\label{SS:steinberg}
The root elements of the Chevalley group $G(\Phi,\B k)$ satisfy the
following relations:
\begin{align*}
x_{\alpha}(s)x_{\alpha}(t) &= x_{\alpha}(s+t)\\
\left[ x_{\alpha}(s),x_{\beta}(t) \right] &= \prod_{i,j>0} x_{i\alpha+j\beta}(C_{i,j}t^is^j)
\end{align*}
where the product is taken in the increasing order of $i+j>0$ and
$C_{i,j}\in \{\pm 1,\pm 2,\pm 3\}$ \cite[Lemma 32.5, Propositions 33.3--5]{MR0396773}.

\subsection{Semisimple elements of Chevalley groups}\label{SS:semisimple}
Besides the root elements one also has semisimple elements of $G(\Phi,\B k)$
which are defined as follows. Let $\alpha\in\Phi$ be a root, $0\neq t\in\B k$ and set 
$$
h_{\alpha}(t):=x_{\alpha}(t)x_{-\alpha}(-t^{-1})x_{\alpha}(t)x_{\alpha}(-1)x_{-\alpha}(1)x_{\alpha}(-1).
$$
They are related to the root elements as follows.
Let $\alpha,\beta\in\Phi$ be roots and let
$\langle\beta,\alpha\rangle:=2\frac{(\beta,\alpha)}{(\alpha,\alpha)}$ be a
corresponding Cartan integer. Then the following
equation holds 
$$
h_{\alpha}(t)x_{\beta}(u)h_{\alpha}(t)^{-1}=x_{\beta}(t^{\langle\beta,\alpha\rangle}u).
$$ 
where $u,t\in\B k$ and $t\neq 0$ \cite[Lemma~20(c)]{MR3616493}. 
\subsection{$S$-arithmetic Chevalley groups}\label{SS:S-chevalley}
Let $G(\Phi,\B k)$ be a Chevalley group over a number field $\B k$.
We consider the $S$-arithmetic group $G(\Phi,\C O_S)$ 
over the ring of $S$-integers $\C O_S\subset \B k$.
Let $E(\Phi,\C O_S)\subset G(\Phi,\C O_S)$ be the subgroup generated by the
root elements $x_{\alpha}(t)$, where $\alpha \in \Phi$ and $t\in \C O_S$.  It
is known that it is a subgroup of finite index.  
We shall call it an {\em elementary} Chevalley group.


\subsection{Tavgen$^{\prime}$s theorem}\label{SS:tavgen}
Let $\Phi$ be an irreducible root system of rank at least $2$.  If $G=E(\Phi,\C O_S)$ is an
elementary $S$-arithmetic Chevalley group then there exists a number $m\in \B N$ such that
every element $g\in G$ can be written as
$g = x_{\alpha_1}(t_1) \cdots x_{\alpha_m}(t_m)$, for some $t_i \in \C O_S$
\cite[Theorem~A]{MR1044049}. We say that $G$ has {\em bounded
generation with respect to root elements}. 

\section{Proof}

Let us recall the following observation which will be frequently used in
this section.
If a group $H$ contains a finite index bounded subgroup $H'$, then $H$
itself is bounded. Indeed, for any conjugation-invariant norm $\nu$ defined on
$H$ one has for a given finite set of representatives $T$ of left cosets of
$H'$ in $H$ that 
$$
\OP{diam}(H,\nu)\leq \OP{diam}(H',\nu|_{H'})+\max\{\nu(t)|t\in T\}<\infty.
$$ 
This implies that it suffices to prove the statement for $G=E(\Phi,\C O_S)$, 
the elementary Chevalley group instead of the full one. Next observe that each
finite index subgroup $H$ of $G$ contains a normal subgroup $H'$ of finite
index.  Thus it suffices to consider only the case $H$ of finite index and
normal in $G$.

\begin{lemma}\label{L:bdness arithm root-groups}
Let $G=E(\Phi,\C O)$ be an arithmetic Chevalley group of rank at least $2$
constructed from the irreducible root system $\Phi$ and $H$ a finite index
normal subgroup of $G$.  Further let $\alpha$ be a root and $\nu$ a
conjugation-invariant norm on $H$.  Then the group $\{x_{\alpha}(a)|a\in\C
O\}\cap H$ is bounded with respect to $\nu$. 
\end{lemma}

\begin{proof}
Let $\xi_0=1,\xi_1,\ldots, \xi_r\in \C O$ be a basis of $\C O$ over $\B Z$,
i.e. $\C O$ splits as a direct product $\bigoplus_{i=0}^r\B Z\xi_i$.
Let $p\in\B N$ be the smallest positive integer such that 
the elements $x_{\delta}(p\xi_l)\in H$ for all $\delta\in\Phi$ and $0\leq l\leq r$.
For $a\in\C O$ observe that there are integers $m_0,\dots,m_r$ such that $a=m_0+m_1\xi_1+\dots+m_r\xi_r$. 
Using division with remainder we can find integers 
$n_l$ and $r_l$ with $0\leq r_l<p$ such that $m_l=pn_l+r_l$ for $0\leq l\leq r.$ But then we get 
$$
x_{\alpha}(a)=(x_{\alpha}(p\xi_0)^{n_0}\cdots x_{\alpha}(p\xi_r)^{n_r})x_{\alpha}(r_0\xi_0+\dots+r_r\xi_r).
$$
Observe that for the second factor there are only finitely many possibilities and thus
it suffices to show that the cyclic subgroup generated by $x_{\alpha}(p\xi_l)$ for $0\leq l\leq r$ 
is bounded with respect to $\nu$. Using the same division with remainder trick for these subgroups again, 
it actually suffices to find a non-zero multiple $v$ of $p$ such that the cyclic subgroup generated by 
$x_{\alpha}(v\xi_l)$ for $0\leq l\leq r$ is bounded with respect to $\nu$. 
In the following let $\xi$ be one of the $\xi_l.$

There exists a subsystem $\Psi\subset \Phi$
isomorphic to one of  $\B A_2$, $\B B_2$ or $\B G_2$ such that $\alpha \in \Psi$.
Moreover, if $\alpha \in \B A_2$ there exist $\beta,\gamma\in \B A_2$ such that 
$\alpha = \beta + \gamma$ and that no other positive combination
of $\beta$ and $\gamma$ is a root. The same holds if $\alpha \in \B B_2$
or $\alpha \in \B G_2$ is a long root.
It follows from~(\ref{SS:steinberg})
that 
$$
x_{\alpha}(p\xi)^{Cpn} 
= \left[ x_{\beta}(p\xi),x_{\gamma}(pn) \right]
$$
for any $n\in \B Z$ and some $C\in \{\pm 1,\pm 2,\pm 3\}$.
In particular, 
$$
\nu(x_{\alpha}(Cp^2\xi)^n) \leq 2\nu(x_{\beta}(p\xi))
$$
which implies that $\nu(x_{\alpha}(p\xi)^n)$ is bounded if $\alpha $ is either a root
contained in $\B A_2$ or a long root in $\B B_2$ or $\B G_2$. 
In particular, this implies that the set $\{x_{\alpha}(pa)^n|n\in\B N\}$ is
bounded for all $a\in\C O$ if $\alpha$ is a long root in $\B G_2$ or $\B B_2.$

If $\alpha \in \B B_2$ is a short root then there exist $\beta,\gamma\in \B B_2$
such that $\alpha = \beta + \gamma$ and $\beta + 2\gamma$ is a long root
and no other positive combination of $\beta$ and $\gamma$ is a root.
Applying (\ref{SS:steinberg}) again, we get that
\begin{align*}
\left[ x_{\beta}(p\xi),x_{\gamma}(pn) \right] &=
x_{\alpha}(p^2\xi Cn) x_{\beta+2\gamma}(p^3\xi C'n^2)\\
x_{\alpha}(p\xi)^{Cpn} &= 
\left[ x_{\beta}(p\xi),x_{\gamma}(pn) \right] 
x_{\beta+2\gamma}(-p^2\xi)^{C'pn^2}.
\end{align*}
which implies that 
$$\nu(x_{\alpha}(p^2C\xi)^n)\leq2\nu(x_{\beta}(p\xi))+\nu(x_{\beta+2\gamma}(p\xi)^{-C'p^2n^2}).$$ 
However note that $\beta+2\gamma$ is a long root and hence 
we already know that $\nu(x_{\beta+2\gamma}(p\xi)^{-C'p^2n})$ is bounded independently from $n.$

Now if $\alpha\in\B G_2$ is a short root, then there are $\beta,\gamma\in\B G_2$ such that $\beta+\gamma=\alpha$ and further 
$\beta+2\gamma$ and $2\beta+\gamma$ are both long roots and no other positive combination of $\beta$ and $\gamma$ are roots. 
Applying (\ref{SS:steinberg}) again, we get that
\begin{align*}
\left[ x_{\beta}(p\xi),x_{\gamma}(pn) \right] &=
x_{\alpha}(p^2\xi Cn) x_{\beta+2\gamma}(p^3\xi C'n^2)x_{2\beta+\gamma}(p^3\xi^2 C''n)\\
x_{\alpha}(p\xi)^{Cpn} &= 
\left[ x_{\beta}(p\xi),x_{\gamma}(pn) \right] 
x_{\beta+2\gamma}(-p^3\xi C'n^2)x_{2\beta+\gamma}(-p^3\xi^2 C''n).
\end{align*}
which implies that 
$$\nu(x_{\alpha}(p^2C\xi)^n)\leq2\nu(x_{\beta}(p\xi))+\nu(x_{\beta+2\gamma}(p\xi)^{-C'p^2n^2})+\nu(x_{2\beta+\gamma}(p\xi^2)^{-p^2C''n}).$$ 
Now the terms $\nu(x_{\beta+2\gamma}(p\xi)^{-C'p^2n^2})$ and $\nu(x_{2\beta+\gamma}(p\xi^2)^{-p^2C''n})$ are bounded as $\beta+2\gamma$ and $2\beta+\gamma$
are long roots in $\B G_2$. This concludes the proof.
\end{proof}

\begin{lemma}\label{L:bdness S-arithm root-groups}
Let $G=E(\Phi,\C O_S)$ be an S-arithmetic Chevalley group of rank at least $2$
constructed from the irreducible root system $\Phi$ and $H$ a finite index
normal subgroup of $G$. Further let $\nu$ be a conjugation-invariant norm on
$H$.  Then the set $$\{yx_{\alpha}(a)y^{-1}|a\in\C O_S,y\in
G,\alpha\in\Phi\}\cap H$$ is bounded with respect to $\nu$. 
\end{lemma}

\begin{proof}
We start with the observation that there exists an element $u\in \C O$ such
that $\C O_S = \C O[u^{-1}]$. Indeed, let $p_i\subseteq \C O$ be the ideal
corresponding to a non-archimedean valuation $v_i \in S$. Let $U = \prod_i
p_i$.  Since the number of classes of ideals of $\C O$ is finite, there exists
$n\in \B N$ and $u\in \C O$ such that $U^n  = (u)$. Then for any $x\in \C O_S$
we have that $u^kx \in \C O$ for some $k\in \B N$ which implies the claim.

Since there are only finitely many elements in $\Phi$, it suffices to
show the boundedness of the set 
$$
\{yx_{\alpha}(a)y^{-1}|a\in\C O_S,y\in G\}\cap H
$$
for a given root $\alpha.$ Since $\C O_S=\C O[u^{-1}]$, for
each element $a\in\C O_S$ there exists $k\in\B N$ and $b\in\C O$ such that
$a=u^{-k}b.$ By $\langle\alpha,\alpha\rangle=2$ and (\ref{SS:semisimple}) it follows that
$h_{\alpha}(u^{-k})x_{\alpha}(u^kb)h_{\alpha}(u^k)=x_{\alpha}(u^{-k}b)$ holds. 
Since $H$ is normal, it suffices to show boundedness for the set 
$$
\{yx_{\alpha}(a)y^{-1}|a\in\C O,y\in G\}\cap H.
$$
Let $T\subset G$ be a finite set of representatives of right cosets of $H$ in~$G$,
so that each $y\in G$ can be written as $y=ht$ for some $h\in H$ and $t\in T$. 
If $yx_{\alpha}(a)y^{-1}\in H$ for $a\in\C O$ then
$$
\nu(yx_{\alpha}(a)y^{-1})=\nu(h(tx_{\alpha}(a)t^{-1})h^{-1})=\nu(tx_{\alpha}(a)t^{-1}).
$$
Since the set $T$ is finite, it suffices to show that
$\{tx_{\alpha}(a)t^{-1}|a\in\C O\}\cap H$ is bounded for any given $t\in T.$
As $H$ is normal, the function 
$$
\nu_t\colon H\cap E(\Phi,\C O)\to \B R_{\geq0},
$$
$x\mapsto\nu(txt^{-1})$ defines a conjugation invariant norm on the finite
index, normal subgroup $H\cap E(\Phi,\C O)$ of $E(\Phi,\C O)$. Thus
Lemma~\ref{L:bdness arithm root-groups} yields that
$$
\{tx_{\alpha}(a)t^{-1}|a\in\C O\}\cap H
=t(\{x_{\alpha}(a)|a\in\C O\}\cap (H\cap E(\Phi,\C O)))t^{-1}
$$ 
is indeed bounded with respect to $\nu.$ 
\end{proof}

\subsection*{Proof of Theorem}
We first prove the special case for $G$ constructed from an irreducible
root system $\Phi$.

For each $\alpha\in\Phi$ let $E_{\alpha}:=\{x_{\alpha}(a)|a\in\C O_S\}$ be given. 
The group $H\cap E_{\alpha}$ has finite index in $E_{\alpha}$, 
so let $T_{\alpha}$ be a finite set of representatives of right cosets of
$H\cap E_{\alpha}$ in $E_{\alpha}.$ Choose $h\in H$. By \cite[Theorem
A]{MR1044049} there is an $m\in\B N$ (independent from $h$), $\alpha_i\in\Phi$
and $e_i\in E_{\alpha_i}$ (dependent on $h$) such that $h=e_1 \cdots e_m$.  
For each $i\leq m$ let $h_i\in E_{\alpha_i}\cap H$ and $t_i\in T_{\alpha_i}$ be
given such that $e_i=h_it_i.$ Notice that 
\begin{align*}
h &= e_1 \cdots e_m=\prod_{i=1}^m h_it_i\\
&=\left(
\prod_{i=1}^m \left( t_1\cdots t_{i-1} \right)
h_{i}
\left( t_1\cdots t_{i-1} \right)^{-1}
\right)
 \cdot
\left( t_1\cdots t_{m} \right)\\
\end{align*}
By Lemma~\ref{L:bdness S-arithm root-groups} 
the elements $\left( t_1\cdots t_{i-1} \right)h_{i}\left( t_1\cdots t_{i-1} \right)^{-1}$ 
are bounded independently of $h$, also $m$ is independent of $h$, thus the element 
$$\prod_{i=1}^m (t_1\cdots t_{i-1})h_{i}(t_1\cdots t_{i-1})^{-1}$$ 
is bounded independently of $h$. 
Furthermore, $\Phi$ is finite, all the sets $T_{\alpha}$ are finite, thus there
are only finitely many possibilities for the product $t_1\cdots t_{m}$
(independently of $h$) and hence $h$ itself is uniformly bounded.
This finishes the proof for the irreducible case.

In the general case let
$\Phi_1,\dots,\Phi_m$ be the irreducible components of $\Phi$
and let $G_i = G(\Phi_i,\C O_S)$.
Then the group $G$ admits the following finite central extension 
$$
1\to Z\to\prod_{i=1}^m G_i\stackrel{\pi}\longrightarrow G\to 1,
$$
where $Z\subset Z\left(\prod_{i=1}^m G_i\right)$ is finite.
Hence $H$ admits the following finite central extension
$$
1\to Z\cap\pi^{-1}(H)\to\pi^{-1}(H)\to H\to 1. 
$$
Since the quotient of a bounded group is bounded, it suffices to show
that $\pi^{-1}(H)$ is bounded.

Consider the subgroups $H_i:=\pi^{-1}(H)\cap G_i$ for $1\leq i\leq m$ and let
$H':=\prod_{i=1}^m H_i.$ Since $H_i$ is of finite index in $G_i$ for each
$i=1,\ldots, m$, $H'$ has finite index in $\pi^{-1}(H)$. Moreover, each $H_i$
is bounded, according to the first part of the proof which implies that $H'$ is
bounded. Consequently, $\pi^{-1}(H)$ is bounded which finishes the proof.
\qed

\begin{remark}
The assumption on the components of $\Phi$ to have rank at least~$2$ is
necessary. Namely the group $\OP{SL}_2(\B Z)=\B Z/4*_{\B Z/2}\B Z/6$ 
is known to be unbounded.
\end{remark}

\bibliography{/home/kedra/sync/bibliography}

\def\polhk#1{\setbox0=\hbox{#1}{\ooalign{\hidewidth
  \lower1.5ex\hbox{`}\hidewidth\crcr\unhbox0}}}
  \def\polhk#1{\setbox0=\hbox{#1}{\ooalign{\hidewidth
  \lower1.5ex\hbox{`}\hidewidth\crcr\unhbox0}}}
  \def\polhk#1{\setbox0=\hbox{#1}{\ooalign{\hidewidth
  \lower1.5ex\hbox{`}\hidewidth\crcr\unhbox0}}}
  \def\polhk#1{\setbox0=\hbox{#1}{\ooalign{\hidewidth
  \lower1.5ex\hbox{`}\hidewidth\crcr\unhbox0}}}
  \def\polhk#1{\setbox0=\hbox{#1}{\ooalign{\hidewidth
  \lower1.5ex\hbox{`}\hidewidth\crcr\unhbox0}}}
  \def\polhk#1{\setbox0=\hbox{#1}{\ooalign{\hidewidth
  \lower1.5ex\hbox{`}\hidewidth\crcr\unhbox0}}} \def\cprime{$'$}
\begin{thebibliography}{1}

\bibitem{eom-chevalley}
Chevalley group.
\newblock {\em Encyclopedia of Mathematics}.
\newblock Available at {\tt
  https://www.encyclopediaofmath.org/index.php/Chevalley\_group}.

\bibitem{MR2819193}
{\'S}wiatos{\l}aw~R. Gal and Jarek K\k{e}dra.
\newblock On bi-invariant word metrics.
\newblock {\em J. Topol. Anal.}, 3(2):161--175, 2011.

\bibitem{MR0396773}
James~E. Humphreys.
\newblock {\em Linear algebraic groups, corrected fifth printing}.
\newblock Springer-Verlag, New York-Heidelberg, 1975.
\newblock Graduate Texts in Mathematics, No. 21.

\bibitem{MR1090825}
G.~A. Margulis.
\newblock {\em Discrete subgroups of semisimple {L}ie groups}, volume~17 of
  {\em Ergebnisse der Mathematik und ihrer Grenzgebiete (3) [Results in
  Mathematics and Related Areas (3)]}.
\newblock Springer-Verlag, Berlin, 1991.

\bibitem{MR3616493}
Robert Steinberg.
\newblock {\em Lectures on {C}hevalley groups}, volume~66 of {\em University
  Lecture Series}.
\newblock American Mathematical Society, Providence, RI, 2016.
\newblock Notes prepared by John Faulkner and Robert Wilson, Revised and
  corrected edition of the 1968 original [ MR0466335], With a foreword by
  Robert R. Snapp.

\bibitem{MR1044049}
O.~I. Tavgen{\cprime}.
\newblock Bounded generability of {C}hevalley groups over rings of
  {$S$}-integer algebraic numbers.
\newblock {\em Izv. Akad. Nauk SSSR Ser. Mat.}, 54(1):97--122, 221--222, 1990.

\end{thebibliography}
\bibliographystyle{plain}

\end{document}